%% file: nonlin_potential_2019_04m_05d.tex
\documentclass[final,12pt]{amsart}
\usepackage{amssymb, amsmath, amsthm,geometry}
\usepackage{mathrsfs}  
\usepackage{graphicx}
\usepackage[colorlinks=true, citecolor=blue, linkcolor=blue, urlcolor=blue]{hyperref}
\newgeometry{asymmetric, centering}
 \usepackage{xcolor}

\usepackage{ifdraft}
\ifoptionfinal{
\usepackage[disable]{todonotes}
}{
\usepackage[norefs, nocites]{refcheck}

\usepackage[notref, notcite]{showkeys}
\usepackage[bordercolor=white, color=white]{todonotes}
}

\makeatletter\providecommand\@dotsep{5}\def\listtodoname{List of Todos}\def\listoftodos{\hypersetup{linkcolor=black}\@starttoc{tdo}\listtodoname\hypersetup{linkcolor=blue}}\makeatother

\newtheorem{lemma}{Lemma}

\newtheorem{theorem}{Theorem}

\newtheorem{definition}{Definition}
\theoremstyle{remark}

\newcommand{\bel}{\begin{equation} \label}
\newcommand{\ee}{\end{equation}}
\def\beq{\begin{equation}}
\def\eeq{\end{equation}}
\newcommand{\bea}{\begin{eqnarray}}
\newcommand{\eea}{\end{eqnarray}}
\newcommand{\beas}{\begin{eqnarray*}}
\newcommand{\eeas}{\end{eqnarray*}}
\newcommand{\pd}{\partial}

\renewcommand{\div}{\mathrm{div}\,}  

\def\R{\mathbb R}

\def\M{\mathcal M}
\def\g{\bar g}

\def\CI{\mathcal C}
\renewcommand{\leq}{\leqslant}
\renewcommand{\geq}{\geqslant}

\def\p{\partial}
\newcommand\rotom{\mho}
\DeclareMathOperator{\Tr}{Tr}
\DeclareMathOperator{\spn}{span}
\DeclareMathOperator{\supp}{supp}

\date{Compiled \today}
\title
{Recovery of zeroth order coefficients in non-linear wave equations}

\author[A. Feizmohammadi]{Ali Feizmohammadi}
\address{Department of Mathematics, University College London, 
Gower Street, London UK, WC1E 6BT.}
\email{a.feizmohammadi@ucl.ac.uk}

\author[L. Oksanen]{Lauri Oksanen}
\address{Department of Mathematics, University College London, 
Gower Street, London UK, WC1E 6BT.}
\email{l.oksanen@ucl.ac.uk}

\begin{document}
\maketitle
\begin{abstract}
This paper is concerned with the resolution of an inverse problem related to the recovery of a scalar (potential) function $V$ from the source to solution map of the semi-linear equation $(\Box_{\g}+V)u+u^3=0$ on a globally hyperbolic
Lorentzian manifold $(\M,\g)$. We first study the simpler model problem where the geometry is the Minkowski space and prove the uniqueness of $V$ through the use of geometric optics and a three-fold wave interaction arising from the cubic non-linearity. Subsequently, the result is generalized to globally hyperbolic Lorentzian manifolds by using Gaussian beams.
\end{abstract}
\tableofcontents

\section{Introduction}

Kurylev, Lassas and Uhlmann introduced an approach to solve inverse coefficient determination problems for non-linear hyperbolic equations in \cite{KLU_invent}.
The approach is based on considering multi-parameter families of solutions, and simultaneous linearizations with respect to each of the parameters. If only a one-parameter family of solutions is employed, the linearization yields simply a solution to the linearized version of the non-linear hyperbolic equation under consideration. However, simultaneous linearizations cause solutions to the linearized equation to interact in a non-linear manner, and 
this leads to richer dynamics in propagation of singularities (or wave packets) than in the case of linear hyperbolic equations. 

In particular, it is possible to construct multi-parameter families of solutions so that once linearized suitably, the corresponding solution to the linearized equation has propagating singularities that originate far away from the sources that generated the family. This then allows for solving inverse coefficient determination problems for non-linear hyperbolic equations in much more general geometric context than what is currently within our reach for linear equations \cite{Eskin, FIKO, KKL}. 

In \cite{KLU_invent} the approach was applied to determination of the conformal class of the Lorentzian metric tensor giving the leading order coefficients in a wave equation with quadratic non-linearity. 
The recovery of leading order coefficients has been considered also in the context of Einstein equations in \cite{KLOU} and subsequently in
\cite{LWUI,UW}. We mention also \cite{WZ} where the leading order coefficient were recovered in the presence of a quadratic derivative nonlinearity. The first order coefficients were considered in 
\cite{CLOP}, and in the present paper we show that the zeroth order coefficients can be recovered as well. 

We modify the approach \cite{KLU_invent} substantially by using systematically wave packets, and thus avoid the use of microlocal analysis present in all the previous works based on this approach.
In the hope that this makes the approach more accessible, we give first the proof in Minkowski geometry, and use classical geometric optics solutions.
Then we proceed to show our main result (Theorem \ref{t1} below) saying, in physical terms, that remote sensing of the potential $V$ in 
$(\Box_{\g}+V)u+u^3=0$ is possible on a globally hyperbolic
Lorentzian manifold $(\M,\g)$.
In this case the wave packets that we use are Gaussian beams. 
The cubic non-linearity is chosen because it leads to the simplest computations using our approach. The cubic choice is discussed in more detail in Section 1.4 of \cite{CLOP}.

Inverse problems for non-linear hyperbolic equations have been under an active study. 
We mention that the approach \cite{KLU_invent}
has also been applied to recovery of coefficients appearing in non-linear terms \cite{LUW}, and to a problem arising in seismic imaging \cite{HUW}. Recently two other approaches were used by Nakamura and Vashisth to recover time-independent leading order coefficients, as well as coefficients in non-linear terms \cite{NV}, and by Kian to recover a general function corresponding to the non-linearity and also including zeroth order coefficients \cite{K}. The latter result is based on a reduction via linearization to the problem to recover the zeroth order coefficient in a linear wave equation. For this reason, contrary to our result, the geometric context in \cite{K} in confined to the cases where results are available for linear wave equations.  

\section{The case of Minkowski geometry} 
\label{sec1}

We consider $\R^{1+n}$, with $n \geq 2$,
and write $(x^0,x^1,\ldots,x^n)=(t,x')=x$ for the Cartesian coordinates.
Let $r,T>0$ and write 
    \begin{align}\label{def_mho}
\rotom=(0,T)\times B(0,r),
    \end{align}
where $B(0,r)$ denotes the ball centered at the origin and radius $r$ in $\R^n$. 
We will formulate an inverse coefficient determination problem with data given on $\mho$.
Let $\kappa>0$ be a fixed sufficiently large integer and define $\mathscr{C}$ as a small neighborhood of the origin in the $\CI^{\kappa}_c(\rotom)$ topology. 
Let $V \in \CI^{\infty}(\R^{1+n})$, and for each $f \in \mathscr{C}$, consider the non-linear wave equation 
\bel{pf1}
\begin{aligned}
\begin{cases}
\Box u + V u +u^3=f, 
&\forall (t,x') \in (0,T)\times \R^n,
\\
u(0,x')= 0,\, \p_t u(0,x')=0,
&\forall x' \in \R^n,
\end{cases}
    \end{aligned}
\ee
where $\Box$ is the d'Alembert operator, that is,
$$
\Box u = \p_t^2 u - \sum_{j=1}^n \p_{x^j}^2 u.
$$
When $\kappa$ is large enough and $\mathscr{C}$ is small enough, there exists a unique solution $u$ to equation \eqref{pf1}. We subsequently define the source to solution map for equation \eqref{pf1} as 
$$ L_V(f) = u|_{\rotom}, \quad \forall f \in \mathscr{C}.
$$

The {\em inverse coefficient determination problem} is to find $V$ given the map $L_V$ up to the natural obstruction given by the finite speed of propagation.

In order to be able to determine $V(p)$ for a point $p \in \R^{1+n}$, there must be a signal, in the form of a non-vanishing solution to \eqref{pf1}, from $\mho$ to $p$ and from $p$ to $\mho$. Due to the finite speed of propagation, a signal from a point $q = (t_0,x_0') \in \R^{1+n}$ can reach only the set 
    \begin{align}\label{future_min}
\mathscr J_{+}(q)=\{(t,x') \in \R^{1+n}\,|\, t \geq t_0,\ |x'-x_0'| \leq t - t_0 \},
    \end{align}
called the future of $q$. We define also the past of $q$ by 
    \begin{align}\label{past_min}
\mathscr J_{-}(q)=\{(t,x') \in \R^{1+n}\,|\, t \leq t_0,\ |x'-x_0'| \leq t_0 - t \},
    \end{align}
and write $\mathscr J_{\pm}(\rotom)=\bigcup_{q \in \rotom} \mathscr J_{\pm}(q)$.
Then $L_V$ contains no information on $V$ outside the causal diamond
$$
\mathbb D := \mathscr J_+(\rotom) \cap \mathscr J_-(\rotom)
= \{ (t,x') \in (0,T) \times \R^{n}\,|\, |x'| \leq r + t,\ |x'| \leq r + T - t \},
$$
see Figure \ref{fig_3pts} in Section \ref{linearmin} below.
On the other hand, we will show the following theorem saying that $L_V$ determines $V$ on $\mathbb D$. 

\begin{theorem}
\label{t0}
Let $L_{V_1},L_{V_2}$ denote the source to solution map for equation \eqref{pf1} subject to functions $V_1,V_2 \in \CI^{\infty}(\R^{1+n})$ respectively. Then:
$$L_{V_1}(f)=L_{V_2}(f) \quad \forall f \in \mathscr{C} \implies V_1=V_2 \quad \text{on}\quad \mathbb D.$$
\end{theorem}
\bigskip

The non-trivial content of the theorem is the remote determination on $\mathbb D \setminus \mho$, as 
it is straightforward to see that $L_V$ determines $V$ on $\mho$.
To see this, let $f \in \CI^\infty_c(\rotom)$
and consider the one-parameter family of sources $f_\epsilon :=\epsilon f$, $\epsilon \in \R$.
For small enough $\epsilon$, 
it holds that $f_\epsilon \in \mathscr C$,
and we let $u_\epsilon$ denote the unique solution to  (\ref{pf1}) subject to this source term. Then $u :=\p_\epsilon u_\epsilon |_{\epsilon=0}$ solves the linear wave equation
\bel{eq4}
\begin{aligned}
\begin{cases}
\Box u + V u =f, 
&\forall (t,x') \in (0,T)\times \R^n,
\\
u(0,x')= 0,\, \p_t u(0,x')=0,
&\forall x' \in \R^n,
\end{cases}
    \end{aligned}
\ee 
and $u|_\mho = \p_\epsilon L_V(f_\epsilon)|_{\epsilon = 0}$.
Observe that if $u(q) \ne 0$ for a point $q \in \mho$, then 
$$
V(q) = \frac{f(q) - \Box u(q)}{u(q)}.
$$
It remains to show that for any $q \in \mho$ there is $f \in \CI^\infty_c(\rotom)$ such that the solution $u$ of (\ref{eq4}) satisfies $u(q) \ne 0$.
But this follows simply by taking any $u \in \CI^\infty_c(\rotom)$ with this property, and setting $f = \Box u + V u$.

Let us also point out that it is an open question if $V|_{\mathbb D}$ is determined by the linearized source to solution map,
$$
\mathcal L_V f = u|_\mho, \quad \forall f \in \CI^{\kappa}_c(\rotom),
$$
where $u$ is the solution of (\ref{eq4}). 
Only in the case that $V(t,x')$ is real-analytic in $t$, this is known to hold due to the variant of the Boundary Control method by Eskin \cite{Eskin}.
The approach to recover time-dependent coefficients in wave equations based on geometric optics, originating from \cite{Stefanov}, fails since no wave packet leaving $\mho$ returns there. For this approach to work, the set on which the data is given (i.e. $\mho$ in our case) needs to enclose the region where $V$ is to be determined (i.e. $\mathbb D$ in our case). 
Like \cite{Stefanov}, the approach in the present section is based on geometric optics, but the difference is that the cubic non-linearity in (\ref{pf1}) allows us to solve the inverse problem with the optimal relation between $\mho$ and $\mathbb D$.

Before entering into the proof of Theorem \ref{t0} in detail, let us briefly explain how the non-linearity is used. Let $f_1, f_2, f_3 \in \CI^\infty_c(\rotom)$ and consider the three-parameter family of sources 
    \begin{align}\label{f_family}
f_\epsilon :=\epsilon_1 f_1 + \epsilon_2 f_2 + \epsilon_3 f_3, \quad \forall \epsilon := (\epsilon_1, \epsilon_2, \epsilon_3) \in \R^3.
    \end{align}
For small enough $\epsilon_1$, $\epsilon_2$ and $\epsilon_3$, 
it holds that $f_\epsilon \in \mathscr C$,
and we let $u_\epsilon$ denote the unique solution to  (\ref{pf1}) subject to this source term. Then 
    \begin{align}\label{def_3lin}
u :=\p_{\epsilon_1} \p_{\epsilon_2} \p_{\epsilon_3} u_\epsilon |_{\epsilon=0}
    \end{align}
solves the linear wave equation (\ref{eq4}) with 
    \begin{align}\label{interaction}
f = -6 u^{(1)}u^{(2)}u^{(3)}, \quad u^{(j)} = \p_{\epsilon_j} u_\epsilon |_{\epsilon=0},
    \end{align}
and $u^{(j)}$ satisfies the same equation with $f=f_j$.
In the proof below, we will use sources $f_j$ that generate geometric optics solutions $u^{(j)}$ with carefully chosen phases. This will allows us to employ $u$ as a highly structured signal from a point $p \in \mathbb D$ back to $\mho$.

The remainder of this section is organized as follows. We start by briefly reviewing the geometric optic solutions for the linearized equation \eqref{eq4} in Section~\ref{GOsolutions}. In Section~\ref{sourcemin} we show 
that sources supported in $\mho$ can be explicitly chosen so that they generate the geometric optics solutions for (\ref{eq4}) that pass through $\mho$, and in Section~\ref{linearmin} we use these sources to construct the three-parameter family (\ref{f_family}). 
Then, in Section~\ref{minrecovery}, we consider the interaction of $u^{(1)}$, $u^{(2)}$ and $u^{(3)}$,
as encoded by $u$ in (\ref{def_3lin}), and conclude the proof of Theorem~\ref{t0}.

The proof in the case of general, globally hyperbolic Lorentzian manifolds reflects the proof in Sections~\ref{sourcemin}--\ref{minrecovery}. Our main theorem is formulated in Section~\ref{lorentzian}, and in Section~\ref{formalgaussian} we present the Gaussian beam construction that replaces the classical geometric optics of Section~\ref{GOsolutions} in the general case. Finally in Section~\ref{mainproof} we perform the analogues of the steps in Sections~\ref{sourcemin}--\ref{minrecovery} in the general case. 

\subsection{Geometric optics}
\label{GOsolutions}

In this section we recall 
the classical construction of approximate geometric optics solutions to the wave equation 
    \begin{align}\label{eq2}
\Box u + Vu =0 \quad \text{in}\quad \R^{1+n}.
    \end{align}
The construction is based on 
the ansatz, 
\bel{goansatz1}
u_\tau(x)=e^{i\tau\xi\cdot x} a_\tau(x)=e^{i\tau\xi\cdot x}\left(\sum_{k=0}^N \frac{a_k(x)}{\tau^k}\right),
\ee 
where $\tau > 0$ is a large parameter, and $\xi \in \R^{1+n}$ and a large integer $N > 0$ are fixed.
Here the notation $\xi\cdot x=\sum_{j=0}^n \xi_j x^j$ is used.
We will view $\xi = (\xi_0, \xi') = (\xi_0, \xi_1, \dots, \xi_n)$ as a co-vector and it needs to be non-zero and light-like, that is to say
$$ |\xi_0|^2=|\xi'|^2:=|\xi_1|^2+\ldots+|\xi_n|^2.$$

We denote by $\xi^{\sharp}$ the vector version of $\xi$ with respect to the Minkowski metric. In other words,
$
\xi^{\sharp} = (-\xi_0, \xi').
$
Let $q = (t_0, x_0')$ be a point in $\R^{1+n}$. 
We will construct the amplitude functions $a_k$, $k=0,1,\dots,N$, so that $u_\tau$ satisfies (\ref{eq2}) up to a remainder term that tends to zero as $\tau \to \infty$ and that $u_\tau$ is supported near the line
    \begin{align}\label{def_gamma}
\gamma_{q,\xi}(s):=s\xi^{\sharp} + q =(-s\xi_0 + t_0, \ s\xi'+x_0'), 
\quad \forall s \in \R.
    \end{align}

As $\xi$ is light-like, it holds that 
\bel{conjugatego}
(\Box+V) (e^{i\tau\xi\cdot x}a_\tau)= e^{i\tau \xi\cdot x} ( -2i\tau\mathcal T_{\xi} a_\tau + (\Box+V) a_\tau),
\ee
where $\mathcal T_{\xi} = -\xi_0\pd_{x_0}+ \sum_{j=1}^n \xi_j \pd_{x_j}$.
The construction of the amplitudes $a_{k}$ is driven by the requirement that the expression \eqref{conjugatego} vanishes in powers of $\tau$. In particular, this imposes the transport equation 
    \begin{align}\label{transp_a0}
\mathcal T_\xi a_0=0
    \end{align}
on $a_0$.
Note that if $\omega \in \R^{1+n}$ satisfies $\xi^\sharp \cdot \omega = 0$, then for any $\chi \in \CI^1(\R)$ it holds that
    \begin{align*}
\mathcal T_{\xi} (\chi(\omega \cdot (x-q)))
= 0.
    \end{align*}

We choose $\omega_j'\in \R^n$ so that the co-vectors 
    \begin{align}\label{axes}
\frac{\xi'}{|\xi'|},\omega'_1,\ldots,\omega_{n-1}'
    \end{align}
form an orthonormal basis for $\R^{n}$ with respect to the Euclidean metric, and write $\omega_j = (0, \omega_j')$. Observe that $\xi^\sharp \cdot \omega_j = 0$ and that 
$$
\{\gamma_{q,\xi}(s)\,|\, s \in \R\}=\{x \in \R^{1+n}\,|\, \xi\cdot (x-q)=\omega_1\cdot (x-q)=\ldots=\omega_{n-1}\cdot (x-q)=0\}.
$$ 
Let $\delta > 0$ and let $\chi_\delta \in \CI_c^\infty((-\delta,\delta))$. We choose
\bel{a_0}
a_0(x)= \chi_\delta(|\xi_0|^{-1} \xi\cdot (x-q)) \prod_{j=1}^{n-1}\chi_{\delta}(\omega_j \cdot (x-q)).
\ee
Then (\ref{transp_a0}) holds and 
    \begin{align}\label{a0_supp}
\supp(a_0(t,\cdot)) \subset H(t,\delta), \quad \forall t \in \R,
    \end{align}
where $H(t,\delta)$ is the hypercube in $\R^n$ with side length $2 \delta$, centred at the unique point $x' \in \R^n$
satisfying $(t,x')=\gamma_{q,\xi}(s)$ for some $s \in \R$, and with the edges pointing to the directions (\ref{axes}).

The subsequent terms $a_{k}$ with $k \geq 1$ are chosen iteratively through the transport equations
\bel{transmin}
-2i\mathcal T_{\xi} a_{k} + (\Box+V) a_{k-1} =0.
\ee
We impose vanishing initial conditions on the hyperplane
$$
\Sigma_{q,\xi} = \{x \in \R^{1+n} \,|\, 
\xi^\sharp \cdot (x-q) = 0 \},
$$
and obtain
    \begin{align}\label{def_ak}
a_{k}(s\xi^{\sharp}+y)=\frac{1}{2i}\int_0^s ((\Box+V)a_{k-1})(\tilde{s}\xi^{\sharp} + y)\,d\tilde{s},
    \end{align}
where $s \in \R$ and $y \in \Sigma_{q,\xi}$.
It follows from (\ref{a0_supp}), via an induction, that also $\supp(a_{k}(t,\cdot)) \subset H(t,\delta)$, and therefore $u_\tau$ is supported near $\gamma_{q,\xi}$.
Moreover, the equations (\ref{transp_a0}) and (\ref{transmin}), together with (\ref{conjugatego}), imply that 
    \begin{align}\label{go_remainder}
\|(\Box+V)u_\tau\|_{\CI^{k}((0,T) \times \R^n)} \lesssim \tau^{-N+k}.
    \end{align}

\subsection{Source terms}
\label{sourcemin}

As in the previous section, let $\xi \in \R^{1+n}$ be non-zero and light-like and let $q = (t_0, x_0') \in \R^{1+n}$. We will assume, furthermore, that $q \in \mho$, and proceed to construct a source $f \in \CI^{\infty}_c(\rotom)$ such that 
the solution to the linear wave equation (\ref{eq4}) is close to the approximate geometric optics solution (\ref{goansatz1}) in a sense that will be made precise below. For this construction to work, it is necessary to require that $\delta > 0$ in (\ref{a_0}) is small enough so that 
    \begin{align}\label{delta_constraint}
H(t_0,\delta) \subset B(0,r),
    \end{align}
cf. (\ref{a0_supp}) and (\ref{def_mho}).

It follows from (\ref{delta_constraint}) and (\ref{a0_supp}) that there exists $\rho > 0$ such that 
$$
\supp(u_\tau(t,\cdot)) \subset B(0,r), \quad \forall t \in (t_0 - \rho, t_0 + \rho).
$$
We choose $\zeta_- \in \CI^\infty(\R)$ such that 
$\zeta_-(t) = 0$ for $t < t_0 - \rho$ and that $\zeta_-(t) = 1$ for $t > t_0$.
Moreover, we choose $\zeta_+ \in \CI^\infty(\R)$
such that $\zeta_+(t) = 0$ for $t > t_0 + \rho$ and that $\zeta_-(t) = 1$ for $t < t_0$.
We are now ready to define the source. Emphasizing the dependence on $\tau$, $q$ and $\xi$, we write
    \begin{align}\label{def_f_qxi}
f_{\tau,q,\xi} = \zeta_+ (\Box + V) (\zeta_- u_\tau) \in \CI^{\infty}_c(\rotom).
    \end{align}

As $\zeta_- = 1$ in the support of $1-\zeta_+$, it holds that
$$
(\Box + V) (\zeta_- u_\tau) - f_{\tau,q,\xi} = (1-\zeta_+)(\Box + V) u_\tau,
$$ 
and (\ref{go_remainder}) implies the estimate
$$
\|(\Box+V)(\zeta_- u_\tau) - f_{\tau,q,\xi} \|_{H^k((0,T)\times \R^n)} \lesssim \tau^{-N+k}.
$$ 
We write $\mathcal U_{\tau} = u$ where $u$ is the solution of the linear wave equation (\ref{eq4}) with the source $f = f_{\tau,q,\xi}$. By combining the above estimate with the usual energy estimate for the wave equation and the Sobolev embedding of $\CI((0,T) \times \R^n)$ in $H^{k+1}((0,T) \times \R^n)$ for $k > (n-1)/2$, we obtain
    \begin{align}\label{mathcalU_estimate}
\|\zeta_- u_\tau - \mathcal U_{\tau} \|_{\CI((0,T)\times \R^n)} \lesssim \tau^{-2},
    \end{align}
when $N \geq k + 2$.

As was shown right after the formulation of Theorem \ref{t0}, the source to solution map $L_V$ determines $V$ on $\mho$.
Recalling the particular form (\ref{def_ak}) of the subleading amplitude functions $a_k$, $k \geq 1$,
we see that they are determined by $L_V$ in 
$$
\supp(u_\tau) \cap (t_0 - \rho, t_0 + \rho) \times \R^{1+n}.
$$
Therefore $f_{q,\xi,\tau}$ is determined by $L_V$.

We will also need a test function whose construction differs from that of $f_{q,\xi,\tau}$ only to the extent that the roles of $\zeta_+$ and $\zeta_-$ are reversed in (\ref{def_f_qxi}). That is, we define 
    \begin{align}\label{def_f_qxi_p}
f_{\tau,q,\xi}^+ = \zeta_- (\Box + V) (\zeta_+ u_\tau) \in \CI^{\infty}_c(\rotom).
    \end{align}
Again $L_V$ determines $f_{\tau,q,\xi}^+$, and the analogue of (\ref{mathcalU_estimate}) reads
    \begin{align}\label{mathcalU_estimate_p}
\|\zeta_+ u_\tau - \mathcal U_{\tau} \|_{\CI((0,T)\times \R^n)} \lesssim \tau^{-2},
    \end{align}
where $\mathcal U_{\tau}$ is now the solution of the linear wave equation 
\bel{eq_back}
\begin{aligned}
\begin{cases}
\Box u + V u = f, 
&\forall (t,x') \in (0,T)\times \R^n,
\\
u(T,x')= 0,\, \p_t u(T,x')=0,
&\forall x' \in \R^n,
\end{cases}
    \end{aligned}
\ee 
with $f=f_{\tau,q,\xi}^+$.

\subsection{Three-parameter family of sources}
\label{linearmin}

Let $p = (t_1, x_1') \in \mathbb D \setminus \mho$.
In this section we will construct a three-parameter family of sources $f_\epsilon$ of the form (\ref{f_family}) so that the cross derivative (\ref{def_3lin}) will act as a structured signal from the point $p$ back to $\mho$. 

As $p \in \mathbb D$, there are $q^- = (t_0, x_0') \in \mho$ and non-zero, light-like $\xi^- \in \R^{1+n}$
such that $t_1 > t_0$ and $p = \gamma_{q^-,\xi^-}(s_0)$ for some $s_0 \in \R$.
Here we are using the notation (\ref{def_gamma}).
We also normalize the co-vector $\xi^- = (\xi_0^-, \dots, \xi_n^-)$ so that $\xi_0^- = -1$. Then $s_0 = t_1 - t_0$.
Using again the fact that $p \in \mathbb D$, 
we see that the point $q^+ := (t_0 + 2 s_0, x_0')$ is in $\mho$.
Moreover, setting $\xi^+ = (-1, -\xi_1^-, \dots, -\xi_n^-)$, it holds that $p = \gamma_{q^+,\xi^+}(-s_0)$, see Figure \ref{fig_3pts}.

\begin{figure}[t]
\def\svgwidth{8cm}
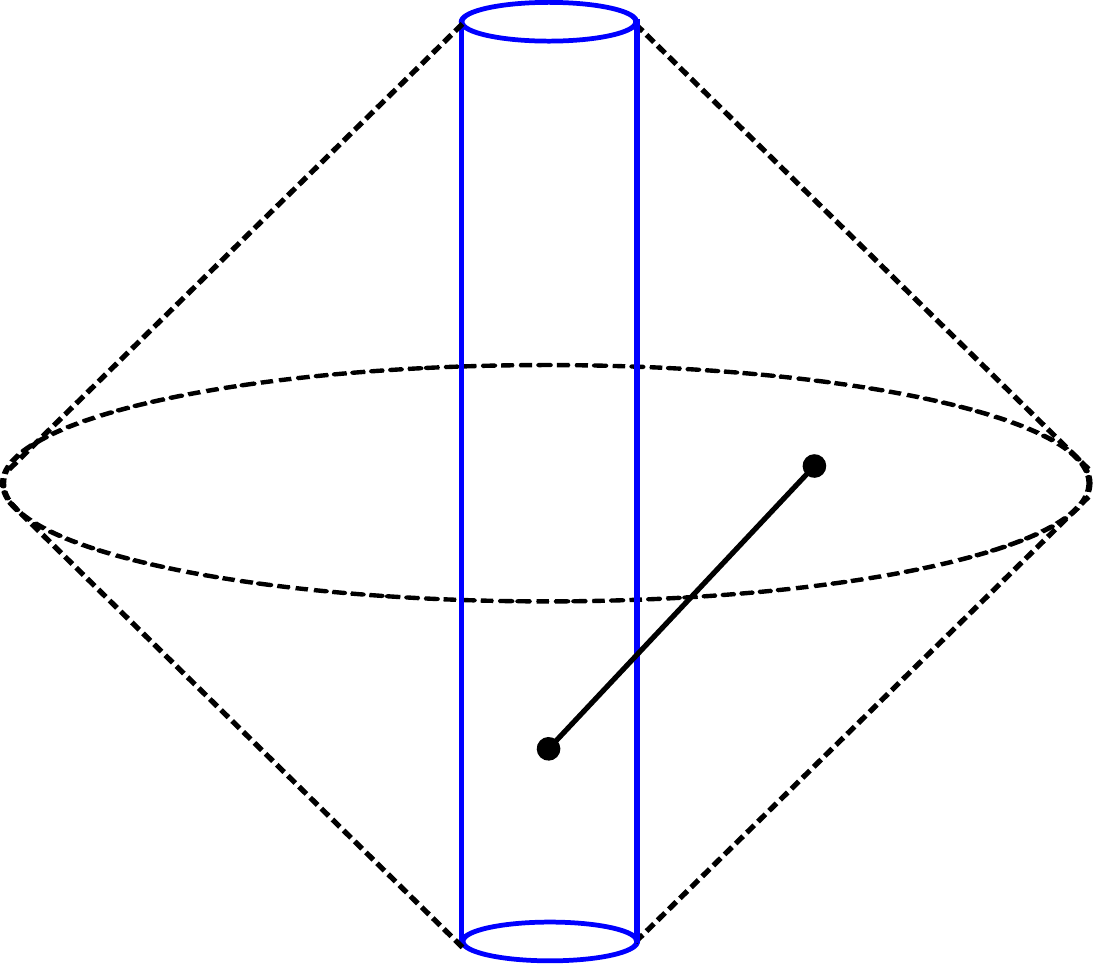
\caption{
The causal diamond $\mathbb D$ is drawn with dashed curves, and $\mho$ is the blue cylinder. The black line segments are on $\gamma_{q^\pm, \xi^\pm}$, joining $q^\pm$ to p. We write also $q^+ = q^{(0)}$ and $q^- = q^{(1)}$, and denote by $q^{(2)}$ and $q^{(3)}$ two perturbations of $q^{(1)}$.
}\label{fig_3pts}
\end{figure}

We will next choose two light-like co-vectors, that are  small perturbations of $\xi^-$, in such a way that $\xi^+$ can be written as a linear combination of $\xi^-$ and the two perturbations. 
We may rotate the coordinate system in the spatial variables $(x^1, \dots, x^n) \in \R^n$, so that in the rotated coordinates, $\xi^+$ and $\xi^-$
are represented by
    \begin{align}\label{def_xi01}
\tilde{\xi}^{(0)}=(-1,-1,\underbrace{0,\ldots,0}_{n-1\,\text{times}}),\quad \tilde{\xi}^{(1)}=(-1,1,\underbrace{0,\ldots,0}_{n-1\,\text{times}}),
    \end{align}
respectively. We define for small $\sigma > 0$
    \begin{align}\label{def_xi23}
\tilde{\xi}^{(2)}=(-1,\sqrt{1-{\sigma}^2},{\sigma},\underbrace{0,\ldots,0}_{n-2\,\text{times}}),\quad \tilde{\xi}^{(3)}=(-1,\sqrt{1-{\sigma}^2},-{\sigma},\underbrace{0,\ldots,0}_{n-2\,\text{times}}),
    \end{align}
and have
\bel{linear dependence}
\sigma^2 \tilde{\xi}^{(0)} + 
\kappa_1 \tilde{\xi}^{(1)} 
+\kappa_2 \tilde{\xi}^{(2)} 
+\kappa_3 \tilde{\xi}^{(3)} = 0    ,
\ee 
where 
    \begin{align}\label{def_kappa}
\kappa_1 = 2(1+\sqrt{1-\sigma^2}) - \sigma^2, \quad \kappa_2 = \kappa_3 = -1 - \sqrt{1-\sigma^2}.
    \end{align}
Finally, we define the co-vector $\xi^{(0)}$ 
to be the representation of $\sigma^2 \tilde{\xi}^{(0)}$, after passing back to the original coordinate system, and $\xi^{(j)}$ for $j=1,2,3$
to be the analogous representations of the co-vectors $\kappa_j \tilde{\xi}^{(j)}$.
Then $\xi^{(0)} = \sigma^2 \xi^+$, $\xi^{(1)} = \kappa_1 \xi^-$ and both $\xi^{(2)}$ and $\xi^{(3)}$
are small perturbations of $-2 \xi^-$.
Note also that $\kappa_1$ is close to $4$.

Define 
    \begin{align}\label{def_qs}
q^{(0)} = \gamma_{p,\xi^{(0)}}(s_0/\sigma^2), \quad 
q^{(j)} = \gamma_{p,\xi^{(j)}}(-s_0/\kappa_j), \quad \text{for $j=1,2,3$}.
    \end{align}
Then $q^{(0)} = q^+$, $q^{(1)} = q^-$ and $q^{(2)}, q^{(3)} \in \mho$ for small enough $\sigma > 0$.
We are now ready to define the following three parameter family of sources 
    \begin{align}\label{f_family_final}
f_{\tau,\epsilon}=\epsilon_1 f_{\tau,q^{(1)},\xi^{(1)}}+\epsilon_2 f_{\tau,q^{(2)},\xi^{(2)}}+\epsilon_3 f_{\tau,q^{(3)},\xi^{(3)}},
    \end{align}
where each $f_{\tau,q^{(j)},\xi^{(j)}}$ is defined by (\ref{def_f_qxi}).

\subsection{Recovery of $V$}
\label{minrecovery}

Let $f_{\epsilon, \tau}$ be as in (\ref{f_family_final}).
Recall that $p$ is an arbitrary point in $\mathbb D \setminus \mho$. In this section, we will prove Theorem \ref{t0} by showing that $V(p)$ is determined by $L_V$. 

For a fixed $\tau > 0$ and small enough $\epsilon_j > 0$
it holds that $f_{\epsilon, \tau} \in \mathscr C$,
and we let $u_{\epsilon, \tau}$ denote the unique solution to  (\ref{pf1}) subject to this source term.
We write $\mathcal U^{(j)}_\tau = \p_{\epsilon_j} u_{\epsilon, \tau}|_{\epsilon = 0}$ for $j=1,2,3$. Then the function $\mathcal U^{(j)}_\tau$ is close, in the sense of the estimate (\ref{mathcalU_estimate}), to the approximate geometric optics solution of the form (\ref{goansatz1}) supported near the line $\gamma_{q^{(j)}, \xi^{(j)}}$.
Moreover, it follows from (\ref{interaction}) that the function
$$
v_\tau = -\frac{1}{6}\frac{\pd^3 u_{\epsilon,\tau}}{\pd \epsilon_1\pd \epsilon_2\pd\epsilon_3}|_{\epsilon=0} 
$$
satisfies the equation
\bel{eq5}
\begin{aligned}
\begin{cases}
\Box v_{\tau} + V v_{\tau}=\mathcal U^{(1)}_\tau \mathcal U_\tau^{(2)}\mathcal U^{(3)}_\tau, 
&\forall (t,x') \in (0,T)\times \R^n,
\\
v_{\tau}(0,x')= 0,\, \p_t v_{\tau}(0,x')=0,
&\forall x' \in \R^n.
\end{cases}
    \end{aligned}
\ee

Recall that $q^{(0)}$ is defined by (\ref{def_qs})
and, modulo a rotation and the rescaling by $\sigma^2$, $\xi^{(0)}$ is defined by (\ref{def_xi01}).
Consider the test function $f_{\tau, q^{(0)}, \xi^{(0)}}^+$ defined by (\ref{def_f_qxi_p}) and denote 
by $\mathcal U^{(0)}_\tau$ the solution of (\ref{eq_back}) with $f = f_{\tau, q^{(0)}, \xi^{(0)}}^+$. As $f_{\tau, q^{(0)}, \xi^{(0)}}^+$ is supported in $\mho$, there holds 
    \begin{align*}
-\frac 1 6 \int_{(0,T)\times \R^n} \p_{\epsilon_1} \p_{\epsilon_2} \p_{\epsilon_3} L_V (f_{\epsilon,\tau})|_{\epsilon = 0} \, f_{\tau, q^{(0)}, \xi^{(0)}}^+ \,dx = \int_{(0,T)\times \R^n} v_\tau \, (\Box + V) \mathcal U^{(0)}_\tau \, dx.
    \end{align*}
After integrating by parts twice, we see that $L_V$ determines the integral
$$
\mathcal I = \int_{(0,T)\times \R^n} \mathcal U^{(0)}_\tau \mathcal U^{(1)}_\tau \mathcal U^{(2)}_\tau \mathcal U^{(3)}_\tau \, dx.
$$

It follows from (\ref{mathcalU_estimate}) and (\ref{mathcalU_estimate_p}) that $\mathcal U^{(j)}_\tau$, with $j=0,1,2,3$, coincides with the corresponding approximate geometric optics solution up to an error of order $\tau^{-2}$.
We denote by $a_k^{(j)}$ the corresponding amplitude functions. 
We will expand $\mathcal I$ in the powers of $\tau$, 
$$
\mathcal I = \mathcal I_0 + \mathcal I_{-1} \tau^{-1} + \mathcal O(\tau^{-2}).
$$
Observe that as $a_0^{(0)}$ is supported near $\gamma_{q^{(0)}, \xi^{(0)}}$ and as $a_0^{(1)}$ is supported near $\gamma_{q^{(1)}, \xi^{(1)}}$, their product is supported near the point $p$, cf. (\ref{def_qs}). For this reason the cut-off functions $\zeta_-$ and $\zeta_+$ in (\ref{mathcalU_estimate}) and (\ref{mathcalU_estimate_p}) do not appear in $\mathcal I_{0}$ and $\mathcal I_{1}$.
Moreover, (\ref{linear dependence}) implies that the phases of the four approximate geometric optics solutions cancel each other under the product in $\mathcal I$. 
Therefore 
$$
\mathcal I_0 = \int_{(0,T)\times \R^n} 
a_0^{(0)}a_0^{(1)}a_0^{(2)}a_0^{(3)} dx,
$$
and analogously, 
\[
\begin{aligned}
\mathcal I_{-1} =
\sum_{|e| = 1}
\int_{(0,T)\times \R^n} a_{e_0}^{(0)} a_{e_1}^{(1)}  a_{e_2}^{(2)}  a_{e_3}^{(3)}\,dx,
\end{aligned}
\]
where $e = (e_0, e_1, e_2, e_3) \in \{0,1,\dots\}^4$ is a multi-index.
In particular, $L_V$ determines $\mathcal I_{-1}$.

Recall that $a_1^{(j)}$ is of the form 
$a_1^{(j)} = b_1^{(j)} + c_1^{(j)}$, where 
for $s \in \R$ and $y \in \Sigma_{q^{(j)}, \xi^{(j)}}$,
    \begin{align}\label{def_b1c1}
b_{1}^{(j)}(s\xi^{(j)\sharp}+y) &= \frac{1}{2i}\int_0^s (\Box a_0^{(j)})(\tilde{s}\xi^{(j)\sharp}+y)\,d\tilde s,
\\\notag
c_{1}^{(j)}(s\xi^{(j)\sharp}+y) &= \frac{1}{2i}\int_0^s (V a_0^{(j)})(\tilde{s}\xi^{(j)\sharp}+y)\,d\tilde s,
    \end{align}
cf. (\ref{def_ak}).
As $a_0^{(j)}$ is independent from $V$, so is $b_1^{(j)}$. Therefore $L_V$ determines the quantity

\[
\begin{aligned}
\mathcal J =&\int_{(0,T)\times \R^n} c_{1}^{(0)} a_0^{(1)}  a_0^{(2)}  a_0^{(3)}\,dx+\int_{(0,T)\times \R^n} a_0^{(0)} {c}_{1}^{(1)}  a_0^{(2)}  a_0^{(3)}\,dx\\
&\quad+\int_{(0,T)\times \R^n} a_0^{(0)} a_0^{(1)}  {c}_{1}^{(2)}  a_0^{(3)}\,dx+\int_{(0,T)\times \R^n} a_0^{(0)} a_0^{(1)}  a_0^{(2)}  {c}_{1}^{(3)}\,dx.
\end{aligned}
\]

For each $j=0,1,2,3$, we let the cut-off function $\chi_\delta$ in the definition of the leading amplitude (\ref{a_0}) to converge to the indicator function of the interval $(-\delta, \delta)$. Then 
$\mathcal J$ converges to 
    \begin{align*}
\mathcal J_\delta = \sum_{j=0}^3 \int_{P_\delta} c^{(j)} dx, \quad c_{1}^{(j)}(s\xi^{(j)\sharp}+y) &= \frac{1}{2i}\int_0^s V(\tilde{s}\xi^{(j)\sharp}+y)\,d\tilde s,
    \end{align*}
where $P_{\delta}$ is a small polygonal neighbourhood of the point $p$. Moreover, 
$$
\lim_{\delta \to 0} \frac {J_\delta}{|P_\delta|} = \sum_{j=0}^3 c^{(j)}(p).
$$
Due to (\ref{def_qs}) we have 
$$
2 i c^{(0)}(p) 
= 
\int_0^{-s_0/\sigma^2} V(\tilde{s}\xi^{(0)\sharp}+q^{(0)})\,d \tilde s
= 
\sigma^{-2}\int_0^{s_0} 
V(-s\xi^{+\sharp}+q^{(0)})\,d s,
$$
and 
$$
2 i c^{(j)}(p) 
= 
\int_0^{s_0/\kappa_j} V(\tilde{s}\xi^{(j)\sharp}+q^{(j)})\,d \tilde s, \quad \forall j=1,2,3.
$$
Recalling the explicit dependence (\ref{def_kappa}) of $\kappa_j$ on $\sigma$, we see that 
$$
2i \lim_{\sigma \to 0} \sigma^2 \sum_{j=0}^3 c^{(j)}(p) = \int_0^{s_0} 
V(-s\xi^{+\sharp}+q^{(0)})\,d s.
$$
Differentiating with respect to $s_0$ shows that $V(p)$ is determined by $L_V$.
This concludes the proof of Theorem \ref{t0}.

\section{The case of globally hyperbolic Lorentzian geometries}
\label{lorentzian}
The rest of this paper is concerned with generalizing 
Theorem \ref{t0} to more general 
Lorentzian geomteries. We begin by reviewing some key concepts from Lorentzian geometry, 
following the notations and definitions in \cite{O'neil}.

Let us consider a smooth $1+n$ dimensional Lorentzian manifold $(\M,\g)$, with $n \geq 2$.
The metric tensor $\g$ is taken to have the signature  $(-,+,\ldots,+)$, and writing $\langle v,w\rangle_{\g}=\sum_{i,j=0}^n \g_{ij}v^iw^j$
for vectors $v,w \in T_p \M$, $p \in \M$,
we recall that 
\begin{itemize}
\item[]{$v$ is {\em time-like} (resp. {\em space-like}) if $\langle v,v\rangle_{\g}<0$ (resp. $>0$),}
\item[]{$v$ is {\em light-like} (or {\em null}) if $\langle v,w\rangle_{\g}=0$.}
\end{itemize} 
The manifold $\M$ is assumed to be time-orientable
in the sense that there exists a globally defined, smooth time-like vector field $Z$ on $M$.
A curve $\alpha$ on $\M$ is called {\em causal} if its tangent vector $\dot \alpha$ is time-like or light-like for all points on the curve,
and a causal curve is {\em future-pointing} if $\langle \dot \alpha,Z\rangle_{\g}<0$.
We write $p \preccurlyeq q$ if there exists a causal future-pointing curve from $p$ to $q$ (the case  $p=q$ is allowed). The generalizations of (\ref{future_min}) and (\ref{past_min}) read as 
$$
\mathscr J_{+}(p)=\{q \in \M\,|\, p \preccurlyeq q\}, \quad
\mathscr J_{-}(p)=\{q \in \M\,|\, q \preccurlyeq p\},
\qquad \forall p \in \M.
$$

We will make the typical assumption that $(\M,\g)$ is {\em globally hyperbolic}. This guarantees that the natural, linear wave equation is well-posed on $\M$. There are several equivalent characterizations of global hyperbolicity, and we recall the one in \cite{BS}: there are no closed causal paths on $\M$ and the intersection 
$\mathscr J_{+}(p) \cap \mathscr J_{-}(q)$ is compact for any pair of points $p, q \in \M$.

Global hyperbolicity implies that there is a global splitting in ``time'' and ``space'' in the sense that
$(\M,\g)$ is isometric to $\R\times M$ with metric 
    \begin{align}\label{splitting}
\g=-\beta(t,x')\,dt\otimes\,dt+g(t,x'), \quad \forall t \in \R,\ x' \in M,
    \end{align}
where $\beta$ is a smooth positive function and $g$ is a Riemannian metric on the $n$ dimensional manifold $M$ smoothly depending on the parameter $t$. Moreover, each set $\{t\}\times M$ is a Cauchy hypersurface in $\M$, that is to say, any causal curve intersects it at most once. 

To simplify the notation, we fix a global splitting of the form (\ref{splitting}) and use it throughout the rest of the paper. Analogously to (\ref{def_mho}), we set 
$$
\rotom=(0,T)\times \mathcal O,
$$
where $T> 0$ and $\mathcal O$ is an open, bounded set in $M$.
Analogously with the Minkowski case, we write again $\mathscr J_{\pm}(\rotom)=\bigcup_{q \in \rotom} \mathscr J_{\pm}(q)$ and define
$$
\mathbb D = \mathscr J_+(\rotom) \cap \mathscr J_-(\rotom).
$$

We will denote by $\nabla^{\g}$ the Levi-Civita connection on $\M$ and let $\div_{\g}$ denote the divergence operator on $\M$. The wave (or Laplace-Beltrami) operator $\Box_{\g}$ acting on smooth functions $\CI^{\infty}(\M)$ is subsequently defined through $\Box_{\g} u= -\div_{\g}\nabla^{\g} u$. In local coordinates $x=(t:=x^0,x^1,\ldots,x^n)=(t,x')$, we have
$$ \Box_{\g} u= -\sum_{i,j=0}^n |\g|^{-\frac{1}{2}}\frac{\pd}{\pd x_{i}}\left(|\g|^{\frac{1}{2}}\g^{ij}\frac{\pd u}{\pd x_j} \right).$$

We consider the following Cauchy problem:
\bel{pf}
\begin{aligned}
\begin{cases}
\Box_{\g} u + V u +u^3=f, 
&\forall x \in (0,T)\times M,
\\
u(0,x')= 0,\, \p_t u(0,x')=0,
&\forall x' \in M,
\end{cases}
    \end{aligned}
\ee
where $V \in \CI^{\infty}(\M)$. This equation is well-posed for $f \in \mathscr{C}$, where $\mathscr{C}$ denotes a neighborhood of origin in the $\CI^{\kappa}_c(\rotom)$ topology with $\kappa$ a sufficiently large but explicit constant. In other words, for each $f \in \mathscr{C}$, there exists a unique small solution $u \in H^1((0,T) \times M)$ to \eqref{pf}.

We define the source to solution map $L_V$ associated to the Cauchy problem \eqref{pf} through
\bel{sotsol}
L_V f:= u\,|_{\rotom}, \quad \forall f \in \mathscr{C}.
\ee
As in the Minkowski case,
we are interested in the problem of determining the unknown potential function $V$ on the causal diamond $\mathbb D$, given the map $L_V$. We have the following result.

\begin{theorem}
\label{t1}
The source to solution map $L_V$ determines $V$ on $\mathbb D$ in the sense that 
$$L_{V_1}\,f=L_{V_2}\,f \quad \forall \,f \in \mathscr{C} \quad \implies V_1=V_2 \quad \text{on}\quad \mathbb D.$$ 
\end{theorem}
\bigskip

The rest of this paper is concerned with the proof of this theorem. This is organized as follows. In Section~\ref{formalgaussian} we recall the construction of Gaussian beams which generalizes the Minkowski geometric optic construction to the globally hyperbolic manifold $(\M,\g)$. Section~\ref{mainproof} begins with the construction of the appropriate source terms that produce these Gaussian beams. Next, in Section~\ref{waveinterac2} we consider the geometry of null geodesics and study the intersections of such curves. Finally, in Section~\ref{waveinterac3} we study the interaction of waves corresponding to a three-fold linearization of the semi-linear equation and derive uniqueness of $V$.

\section{Gaussian beams}
\label{formalgaussian}

This section is concerned with the review of (formal) Gaussian beams. Gaussian beams are a classical construction and we refer to \cite{KKL} for a construction in the case of a static Lorentzian manifold which is a product of a time interval with a Riemannian manifold. As the geometry here is not stationary (in the sense that $g,\beta$ depend on $t$), we will explicitly present the construction. 

\subsection{Fermi Coordinates }

In this section, we recall Fermi coordinates (or geodesic coordinates) near a null geodesic $\gamma$, that is, a geodesic with light-like tangent vector $\dot \gamma$. For similar constructions in the context of stationary Lorentzian geometries or Riemannian geometries with a product structure, we refer the reader to \cite{FIKO} and \cite{DKLS} respectively. 


\begin{lemma}[Fermi Coordinates]
\label{fermi}
Let $\delta > 0$, $a < b$ and let $\gamma: (a-\delta,b+\delta) \to \M$ be a null geodesic on $\M$. There exists a coordinate neighborhood  $(U,\Phi)$ of $\gamma([a,b])$, with the coordinates denoted by $(z^0:=s,z^1,\ldots,z^n)$, such that:
\begin{itemize}
\item[(i)] {$\Phi(U)=(a-\delta',b+\delta') \times B(0,\delta')$ where $B(0,\delta')$ denotes a ball in $\mathbb{R}^{n}$ with a small radius $\delta' > 0$.}
\item[(ii)]{$\Phi(\gamma(s))=(s,\underbrace{0,\ldots,0}_{n \hspace{1mm}\text{times}})$}. 
\end{itemize}
Moreover, the metric tensor $\g$ satisfies in this coordinate system  
    \begin{align}\label{g_on_gamma}
\g|_\gamma = 2ds\otimes dz^1+ \sum_{\alpha=2}^n \,dz^\alpha\otimes \,dz^\alpha,
    \end{align}
and $\partial_i \g_{jk}|_\gamma = 0$ for $i,j,k=0,\ldots,n$. Here, $|_\gamma$ denotes the restriction on the curve $\gamma$.
\end{lemma}

\begin{proof}
Write $q=\gamma(a-\delta)$ and $e_0 = \dot{\gamma}(a-\delta)$. Note that $\g(e_0,e_0)=0$. 
There are non-zero $c_0 \in \R$ and $e'_0 \in T_q M$
such that $e_0=c_0\pd_t +e'_0$. We set 
$$
e_1=\frac{1}{\beta c_0^2}(-c_0\pd_t+e'_0).
$$
Then $\g(e_1,e_1)=0$ and $\g(e_0,e_1)=2$. 
Finally, we choose vectors $e_2,\dots e_n \in T_q \M$ such that $\g(e_k,e_k)=1$ for all $k=2,\ldots,n$, and $\g(e_i,e_j)=0$ for all $i=0,1,\ldots,n$ and $j=2,\ldots,n$ with $i \neq j$. Then $e_0,\dots,e_n$ is a pseudo orthonormal basis on $T_q \M$. For each $k=0,\ldots,n$, let $E_k(s) \in T_{\gamma(s)}\M$ denote the parallel transport of $e_k$ along $\gamma$ to the point $\gamma(s)$. Observe that $E_0 = \dot \gamma$. Then $E_0(s),...,E_n(s)$ is a pseudo orthonormal basis on $T_{\gamma(s)}\M$. 

We now define the coordinate system $(z^0:=s,\ldots,z^n)$ through the map 
$$\mathcal F(s,z^1,...,z^n) = \exp_{\gamma(s)}(\sum_{k=1}^{n}z^{k} E_{k}(s)),
$$
where $\exp_{p}:T_p\M \to \M$ denotes the exponential map on $\M$ at a point $p$. Clearly
$$
\mathcal F(s,\underbrace{0,\ldots,0}_{n \hspace{1mm}\text{times}}) = \gamma(s),
\quad \forall s \in (a-\delta,b+\delta),
$$
is injective as $\gamma$ is not self-intersecting due to global hyperbolicity. Furthermore,
$$
\frac{\partial}{\partial z^k} \mathcal F(s,\underbrace{0,\ldots,0}_{n \hspace{1mm}\text{times}})= E_{k}(s),
\quad \forall k=0,\dots,n.
$$
The inverse function theorem applies, and we conclude that $\mathcal F$ is a smooth diffeomorphism in a neighborhood of $(a-\delta, b+\delta) \times \{0\}$. We define $\Phi =\mathcal F^{-1}$ and note that (i) and (ii) are satisfied. 

Since $E_0(s),...,E_n(s)$ is a pseudo orthonormal basis, also (\ref{g_on_gamma}) holds.
Let us now study the derivatives of $\g$ on $\gamma$.
Let $(s,a^1,\ldots a^n) \in \Phi(U)$
be fixed and consider the path 
$h(t)=\exp_{\gamma(s)}(t \sum_{i=1}^n a^{i}E_i(s))$.
As $h$ is a geodesic, it satisfies 
$$ \ddot{h}^{k} + \bar{\Gamma}^{k}_{\alpha \beta} \dot{h}^\alpha \dot{h}^\beta = 0,$$
where $\bar{\Gamma}^{k}_{\alpha \beta}$ are the Christoffel symbols of the second kind for $\g$.
In the Fermi coordinates $h^0 = s$ and $h^i = t a^i$ for $i=1,\dots,n$, and therefore $\ddot h^k = 0$ for all $k=0,\dots,n$.
By varying $(a^1, \dots, a^n)$, we see that 
$\bar{\Gamma}^{k}_{\alpha \beta} = 0$ for $k = 0,\dots,n$ and $\alpha,\beta = 1,\dots,n$.

As $E_\alpha$ is defined as a parallel transport,
there holds
$$
\nabla^{\g}_{\partial_0} \partial_j = \nabla^{\g}_{\dot{\gamma}(s)} E_{j}(s) = 0,
$$
and therefore, using the symmetry of the Levi-Civita connection,
$\bar{\Gamma}^{k}_{0 j} = \bar{\Gamma}^{k}_{j 0} = 0$
for $k,j = 0,\dots,n$.
Thus all the Christoffel symbols $\bar{\Gamma}^{k}_{i j}$, $i,j,k=0,\dots,n$, vanish on $\gamma$.
Hence, there holds on $\gamma$,
$$ 
\partial_{k} \g_{ij} = \langle \nabla^{\g}_{\partial_k} \partial_i, \partial_j \rangle_{\g} + \langle \partial_i, \nabla^{\g}_{\partial_k} \partial_j \rangle_{\g} = \Gamma_{k i j} + \Gamma_{i k j} = 0,
$$
where $\bar\Gamma_{\alpha i j} = \g_{\alpha \beta} \bar{\Gamma}^{\beta}_{i j}$ are the Christoffel symbols of the first kind.
\end{proof}

\subsection{WKB approximation}
\label{WKB}

We use the shorthand notation
$$
\mathcal P_{V} u = (\Box_{\bar{g}} + V)u.
$$
Analogously to the approximate geometric geometric optics solutions in Section \ref{GOsolutions},
we will construct approximate solutions to $\mathcal P_{V} u = 0$ which concentrate on a given null geodesic $\gamma : (a-\delta,b+\delta) \to \M$. 
We write $I = [a-\delta',b+\delta']$ with $\delta'>0$ as in Lemma~\ref{fermi}, and define the tubular set 
$$\mathcal V= \{x \in \M\,|\, s \in I,\ |z'|:=\sqrt{|z^1|^2+\ldots+|z^n|^2}<\delta'\}.$$
We consider the WKB ansatz
$$u_\tau(s,z') = e^{i\tau \phi(s,z')} a_\tau(s,z'),$$
in the Fermi coordinates $z=(s,z')$ near $\gamma$.
The complex valued phase $\phi \in C^{\infty}(\mathcal V)$ and amplitude $a_\tau \in C^{\infty}_c(\mathcal V)$ will be constructed below. 

We have
\bel{conj}
\mathcal P_V (e^{i\tau \phi} a_\tau) = e^{i\tau \phi} \left (\tau^2 (\mathcal{H}\phi) a_\tau - i \tau \mathcal{T}a_\tau + \mathcal P_V a_\tau\right).
\ee
where the operators $\mathcal{H},\mathcal{T}: C^{\infty}(\M) \to C^{\infty}(\M)$ are defined through
\bel{eikonal-transport}
\mathcal{H}\phi:=\langle d\phi, d\phi \rangle_{\bar{g}}, \quad \mathcal{T}a:= 2 \langle d\phi,da\rangle_{\g}- (\Box_{\g} \phi)a.
\ee
We make the following ansatz for $\phi$ and $a$ respectively:
\bel{phase-amplitude}
\begin{aligned}
 \phi = \sum_{j=0}^{N} \phi_j(s,z')& \quad \text{and} \quad a_\tau(s,z')= \chi(\frac{|z'|}{\delta'}) \sum_{k=0}^{N} \tau^{-k}v_k(s,z'),\\
&v_k(s,z')=\sum_{j=0}^{N} v_{k,j}(s,z'),
\end{aligned}
\ee
where for each $j,k=0,\ldots,N$, $\phi_j$ and $v_{k,j}$ are complex valued homogeneous polynomials of degree $j$ with respect to the variables $z^{i}$ with $i=1,...,n$, and $\chi(t)$ is a non-negative smooth function of compact support such that $\chi(t)=1$ for $|t| \leq \frac{1}{4}$ and $\chi=0$ for $|t|\geq \frac{1}{2}$. 

The equation $\mathcal{H}\phi = 0$ is often called the eikonal equation, and we require that it is satisfied in the following sense on $\gamma$,
\bel{eikonal}
\frac{\partial^{\alpha}}{\partial {z}^{\alpha}} (\mathcal{H}\phi)(s,0,\ldots,0) = 0, \quad \forall s \in I, 
\ee
for all multi-indices $\alpha = \{0,1,\dots\}^{1+n}$ with $|\alpha| \le N$. Here $I$ is the interval in Lemma \ref{fermi}.
We also require that the following transport type equations are satisfied on $\gamma$ by the leading $v_0$ and subsequent $v_{k}$, $k=1,\dots,N$, amplitudes,
    \begin{align}
\label{transport}
\frac{\partial^{\alpha}}{\partial {z}^{\alpha}} (\mathcal{T}v_{0})(s,0,\ldots,0) &= 0, \quad \forall s \in I, 
\\\label{transport1}
\frac{\partial^{\alpha}}{\partial {z}^{\alpha}} (-i\mathcal{T}v_{k}+\mathcal P_V v_{k-1})(s,0,\ldots,0) &= 0, \quad \forall s \in I, 
    \end{align}
for all $\alpha = \{0,1,\dots\}^{1+n}$ with $|\alpha| \le N$.
With these notations, we will define:
\begin{definition}
\label{gaussbeamdef}
An approximate Gaussian beam of order $N$ along $\gamma$ is a function $u_\tau=e^{i\tau \phi}a_\tau$ with $\phi,a_\tau$ defined as in \eqref{phase-amplitude}, such that the following properties hold:
\begin{itemize}
\item[(i)] {Equations \eqref{eikonal}--\eqref{transport1} hold.}
\item[(ii)]{$\Im(\phi)|_{\gamma}=0$, that is, the imaginary part of $\phi$ vanishes on $\gamma$.}
\item[(iii)] $\Im(\phi)(z) \geq C |z'|^2$ for all points $z \in \mathcal V$. 
\end{itemize}
\end{definition}

It follows from (i) and (iii) that $u_\tau$ is an approximate solution to equation $\mathcal P_V u = 0$ in the sense of the following lemma. 
\begin{lemma}
\label{lemerror}
Let $u_\tau$ be an approximate Gaussian beam of order $N$ along $\gamma$ in the sense of Definition~\ref{gaussbeamdef}. 
Suppose that the end points of $\gamma$ are outside $[0,T] \times M$ in the sense that $\gamma(a), \gamma(b) \notin [0,T] \times M$.
Then for all $\tau > 0$
$$\| \mathcal P_V u_\tau \|_{H^k((0,T)\times M)} \lesssim \tau^{-K},\quad \| u_\tau\|_{\mathcal C ((0,T)\times M)}\lesssim 1.$$
where $K=\frac{N+1-k}{2}-1$.
\end{lemma}
\begin{proof}
The second estimate follows trivially from \eqref{phase-amplitude} and (iii). 
Note that equations \eqref{eikonal}, \eqref{transport} and \eqref{transport1} imply that
\[ |\p_z^\alpha \mathcal P_V u_\tau | \lesssim \tau^{|\alpha|}|e^{i\tau\phi}|(C_0\tau^2 |z'|^{N+1}+C_1\tau|z'|^{N+1}+C_2 \tau^{-N}).\]
Moreover, $|e^{i\tau\phi}| \le e^{-C\tau|z'|^2}$ by (iii).
Writing $r = |z'|$, we obtain the first estimate by using the change of variables $\rho^2 = \tau r^2$,
$$
\int_\R e^{-C \tau r^2} r^{2k} dr = C_k \tau^{-k/2-1}, 
$$
where $C_k = \int_\R e^{-C \rho^2} \rho^{2k} d\rho$ and $k>0$.
\end{proof}

Observe also that if $u_\tau= e^{i\tau {\phi}} {a}_\tau$ is an approximate Gaussian beam, then also
\begin{align}\label{gb_tilde}
\tilde u_\tau= e^{-i\tau \bar{\phi}} \bar{a}_\tau,
\end{align}
satisfies the estimates in Lemma \ref{lemerror}.

\subsubsection{The phase function}

Let us now construct $\phi$ and $a_\tau$ satisfying \eqref{eikonal}--\eqref{transport1}.
We begin by constructing the expansion of the phase function $\phi$ in such a way that equation \eqref{eikonal} holds. For $|\alpha|=0$, we obtain the equation on $\gamma$
$$
\sum_{k,l=0}^n\bar{g}^{kl} \frac{\partial \phi}{\partial z^k} \frac{\partial \phi}{\partial z^l} = 0.$$
Using (\ref{g_on_gamma}), this reduces to
\bel{m=0}
2 \partial_0 \phi \, \partial_1 \phi + \sum_{k=2}^n (\partial_k \phi)^2 =0.
\ee
Recalling that for all $i,j,k=0,\dots,n$ we have $\pd_i \g^{jk}=0$ on $\gamma$, we obtain 
similarly for $|\alpha|=1$, 
\bel{m=1}
\sum_{k,l=0}^n \bar{g}^{kl} \partial^2_{i k} \phi \, \partial_{l} \phi =0,
\ee
for all $i=1,\ldots,n$. 
Equations \eqref{m=0} and \eqref{m=1} are satisfied setting
\bel{phi01}
\phi_0=0\quad\text{and}\quad \phi_1=z^1.
\ee
Indeed, \eqref{m=0} holds since $\p_0 \phi = \p_k \phi = 0$ for $k=2,\dots,n$, and \eqref{m=1}
holds since $\g^{kl} \p_l \phi \ne 0$ on $\gamma$ only if $k=0$ and $l=1$, and since $\p_{i0}^2 \phi = 0$ on $\gamma$ for all $i=1,\ldots,n$.

Next, we write 
$$\phi_2(s,z') := \sum_{1 \leq i,j \leq n} H_{ij}(s) z^{i}z^{j},$$ 
where $H_{ij}=H_{ji}$ is a complex-valued matrix. By Definition~\ref{gaussbeamdef}, we require that 
the imaginary part of $H$ is positive definite, that is,
\bel{positivity}
\Im H(s) >0, \quad \forall s\in I.
\ee
The equation (\ref{eikonal}) with $|\alpha|=2$ is equivalent with
    \begin{align*}
\sum_{k,l=0}^n(
2 \bar{g}^{kl} \partial^3_{kij} \phi \, \partial_l \phi 
+
2\bar{g}^{kl} \partial^2_{ki} \phi \, \partial^2_{lj}\phi 
+
\partial^2_{ij}\bar{g}^{kl} \partial_k \phi \, \partial_l \phi
+ 
4 \p_i \g^{kl}\p_{jk}^2\phi\,\p_l \phi)=0,
    \end{align*}
for all $i,j=1,\dots,n$. Using (\ref{phi01}), (\ref{g_on_gamma}), $\p_{i0}^2 \phi = 0$, $i=1,\dots, n$, and $\pd_i \g^{kl}=0$, this reduces to 
$$2\bar{g}^{10} \partial^3_{0ij}\phi +2\sum_{k=2}^{n} \partial^2_{ki} \phi \, \partial^2_{kj}\phi + \partial^2_{ij}\bar{g}^{11}=0.$$
Noting that $\partial^2_{ij} \phi = 2 H_{ij}$,
we obtain the following Riccati equation for $H(s)$:  
\bel{riccati}
\frac{d}{ds} H + HCH + D=0, \quad \forall s \in I,
\ee
where $C$ and $D$ are the matrices defined through
\bel{Cmatrix}
\begin{cases}
C_{11}= 0&\\
C_{ii}=2& \quad i=2,\ldots,n, \\
 C_{ij}=0& \quad \text{otherwise,}
\end{cases}
\qquad D_{ij}= \frac{1}{4} \p^2_{ij} \g^{11}.
\ee 

We recall the following result from \cite[Section 8]{KKL} regarding solvability of the Riccati equation:
\begin{lemma}
\label{ricA}
Let $s_0 \in I$ and let $H_0$ be a symmetric matrix with $\Im H_0 > 0$.
The Riccati equation (\ref{riccati}), together with the initial condition $H(s_0) = H_0$, has a unique solution $H(s)$ for all $s \in I$. We have $\Im H>0$ and $H(s)=Z(s)Y^{-1}(s)$, where the matrix valued functions $Z(s),Y(s)$ solve the first order linear system
$$ \frac{d}{ds} Y = CZ\quad \text{and}\quad  \frac{d}{ds} Z = -DY, \quad \text{subject to} \quad Y(s_0)=I,\quad Z(s_0)=H_0.$$ 
Moreover, the matrix $Y(s)$ is non-degenerate on $I$, and there holds
$$
\det(\Im H(s)) \cdot |\det(Y(s))|^2=\det(\Im(H_0)).
$$
\end{lemma}

We refer the reader to \cite[Section 3.5]{FIKO} for a geometrically invariant interpretation of the function $Y(s)$ above. 
With the help Lemma \ref{ricA}, we have so far succeeded in determining the coefficients of $\phi$ up to the third term in  \eqref{phase-amplitude}. 
The remaining terms can be solved through linear first order ODEs. 

We will describe only the case $j=3$ in detail, the cases $j > 3$ being analogous. 
We see that equation (\ref{eikonal}) with $\p_z^\alpha = \p_p\p_q\p_r$ is equivalent to 
$$
2 \sum_{k,l=0}^n(
\bar{g}^{kl} \partial_{k} \p_z^\alpha \phi \, \partial_l \phi 
+
\bar{g}^{kl} \partial^3_{kpq} \phi \, \partial^2_{lr}\phi
+
\bar{g}^{kl} \partial^3_{kpr} \phi \, \partial^2_{lq}\phi
+
\bar{g}^{kl} \partial^3_{kqr} \phi \, \partial^2_{lp}\phi)
+ 
\mathcal F_\alpha = 0,
$$
where $\mathcal F_\alpha$ depends only on $\phi_j$ with $j \leq 2$. It holds on $\gamma$ that 
$$
\sum_{k,l=0}^n 
\bar{g}^{kl} \partial_{k} \p_z^\alpha \phi \, \partial_l \phi
= \p_s \p_z^\alpha \phi,
$$
and we see that the coefficients $\p_z^\alpha \phi$ with $|\alpha|=3$ satisfy a system of linear ODEs with the right-hand side depending on $\phi_j$ and $\p_s \phi_j$ with $j \leq 2$.
Solving this system with any fixed initial condition gives $\p_z^\alpha \phi$ with $|\alpha|=3$, and the polynomials $\phi_j$ of higher degree are constructed analogously.


\subsubsection{The amplitude function}
We study next the leading amplitude function $v_0$ by determining the terms $\{v_{0,k}\}_{k\geq0}$ in such a way that equation \eqref{transport} holds for all $m=0,\ldots,N$. For $|\alpha|=0$, using the definition of $\mathcal{T}$, we obtain on $\gamma$
$$2 \sum_{k,l=0}^n \g^{kl} \frac{\p \phi}{\p z^k}  \, \frac{\p v_0}{\p z^l}  - (\Box_{\g} \phi) v_0=0, \quad \forall s \in I.$$
Recalling Lemma ~\ref{fermi}, we have on $\gamma$
$$-\Box_{\g} \phi= \sum_{i,j=0}^n \g^{ij}\pd^2_{ij}\phi=\sum_{i=2}^n \pd^2_{ii} \phi=\Tr(CH),$$
and therefore
$$ 2 \frac{d}{ds} v_{0,0} + \Tr(CH) v_{0,0}=0, \quad \forall s\in I.$$
Lemma ~\ref{ricA} yields
$$ \Tr(CH)=\Tr(\dot{Y}Y^{-1})=\Tr \frac{d}{ds} \log Y =\frac{d}{ds}\log \det Y,$$
which implies that we can set
\bel{v_0}
v_{0,0}(s)= \det(Y(s))^{-\frac{1}{2}}, \quad \forall s \in I.
\ee

The subsequent terms $v_{0,k}$ with $k=1,\ldots,N$ can be constructed by solving linear first order ODEs. Indeed, taking $m=k$ in equation \eqref{transport} and recalling the definition of $\mathcal{T}$ we obtain the following equation for the homogeneous polynomial $v_{0,k}(s,z')$:
\bel{vk}
2 \frac{\p}{\p s} v_{0,k} + \Tr(CH) v_{0,k} + \mathcal{E}_k=0 \quad \forall k \geq 1,\ s \in I.
\ee
where $\mathcal{E}_k$ is a homogeneous polynomial of degree $k$ in the $z'$ coordinates with the coefficients only depending on $\{v_{0,l}\}_{l=0}^{k-1}$ and $\{\phi_l\}_{l=0}^{k+2}$.

To determine the subsequent terms $v_{j}$ we need to solve equation \eqref{transport1}, but this can be accomplished analogously to the above argument and is therefore omitted for the sake of brevity.
However, let us establish in detail the analogue of (\ref{def_b1c1}) that will be needed later.
Equation (\ref{transport1}) with $k=1$ and $|\alpha| = 0$ reads on $\gamma$,
$$
2 \frac{d}{ds} v_{1,0} + \Tr(CH) v_{1,0}=  \Box_{\g} v_0 + V v_{0,0},
$$
and therefore we may take
\bel{v_1}
\begin{aligned}
v_{1,0}(s)&=b_{1,0}(s)+c_{1,0}(s),\\
b_{1,0}(s)&=-\frac{i}{2} \det Y(s)^{-\frac{1}{2}}\int_{s_0}^s (\Box_{\g}v_0)(\tilde{s},0) \det Y(\tilde{s})^{\frac{1}{2}}\,d\tilde{s}.\\
c_{1,0}(s)&=-\frac{i}{2}\det Y(s)^{-\frac{1}{2}}\int_{s_0}^s V(\tilde{s},0)\,d\tilde{s}.
\end{aligned}
\ee

This completes the construction of solutions $\phi$ and $a_\tau$ to equations \eqref{eikonal}--\eqref{transport1}. 
The function $u_\tau=e^{i\tau \phi}a_\tau$ is then a formal Gaussian pf order $N$ when $\delta'>0$ in (\ref{phase-amplitude}) is sufficiently small. Indeed,
conditions (i) and (ii) in Defition \ref{gaussbeamdef} follow from  \eqref{eikonal}--\eqref{transport1} and (\ref{phi01}), repectively, and (iii) follows from (\ref{positivity}) for small $\delta'>0$.

\section{Proof of Theorem~\ref{t1}}
\label{mainproof}

\subsection{Source terms}
Let $\gamma$ be a null geodesic and suppose that the end points of $\gamma$ are outside $[0,T] \times M$ in the sense of Lemma \ref{lemerror}. Suppose also that $\gamma$ intersects $\mho$ 
and write $\gamma(s_0) = q \in \mho$, $\xi^\sharp = \dot \gamma(s_0)$.
Let $u_\tau$ to be a formal Gaussian beam along $\gamma$.
Analogously to (\ref{def_f_qxi}), we can again choose cut off functions $\zeta_+$ and $\zeta_-$ such that
the solution $\mathcal U_\tau$ of 
\bel{eq6}
\begin{aligned}
\begin{cases}
\mathcal P_V \mathcal U_\tau =f_{\tau,q,\xi}, 
&\forall (t,x') \in (0,T)\times M,
\\
\mathcal U_\tau(0,x')= 0,\, \p_t \mathcal U_\tau(0,x')=0,
&\forall x' \in M.
\end{cases}
    \end{aligned}
\ee 
with the source 
    \begin{align*}
f_{\tau,q,\xi} = \zeta_+ (\Box_{\g} + V) (\zeta_- u_\tau) \in \CI^{\infty}_c(\rotom)
    \end{align*}
satisfies the estimate
\bel{sourceest1}
\|\mathcal U_\tau - \zeta_- u_\tau\|_{H^{k+1}((0,T)\times M)} \lesssim \tau^{-K}.
\ee
Here $K$ is as in Lemma \ref{lemerror} and $\zeta_-=1$ in $\mathscr J_{+}(q)$. 
It is also important that $s_0$ in (\ref{v_1}) is chosen so that 
\begin{align}\label{init_conds_q}
q=\gamma(s_0) = (s_0, 0, \dots, 0)
\end{align}
in the associated Fermi coordinates. Likewise, initial conditions for the other higher order amplitudes $v_{k,j}$, $k \geq 1$, $j \geq 0$,
need to be set at this point. 
As before the map $L_V$ determines $V|_{\rotom}$ uniquely, and $f_{\tau,q,\xi}$ can be then constructed given $L_V$. 
Moreover, we can also construct a test function $f_{\tau,q,\xi}^+$ analogously to (\ref{def_f_qxi_p}).


\subsection{Three-parameter family of sources}
\label{waveinterac2}

We start by recalling a key lemma, see \cite[Lemma 2.3]{KLUI}. We will re-produce the proof for completeness, and toward that end we recall some concepts from Lorentzian geometry, namely, the notions of lengths of causal curves $\alpha:I\to \R$ and time-separation between points $p,q \in \M$. We can define the length $L$ of a causal curve $\alpha:I\to \R$ as follows
$$ L(\alpha)=\int_I \sqrt{-\g(\dot{\alpha}(s),\dot{\alpha})(s)}\,ds.$$
We also define the time-separation function $\tau(p,q) \in [0,\infty)$ for $p\preccurlyeq q$ through
\bel{timeseparation}
\tau(p,q)=\sup \{L(\alpha)\,|\, \alpha \,\text{is a future pointing causal curve from $p$ to $q$}\}.
\ee
We set $\tau$ to be zero if $p \preccurlyeq q$ does not hold. 
Under the global hyperbolicity assumption,
$\tau:\M\times\M\to \R$ is continuous.
Heuristically, time separation in globally hyperbolic Lorentzian geometries plays the role of Riemannian distance in Riemannian geometries. Indeed, we have the well-known proposition that given $p \preccurlyeq q$, there exists a causal geodesic from $p$ to $q$ of length $\tau(p,q)$. 
Furthermore given a curve $\alpha$, we say that a path $\alpha$ on $\mathcal M$ is a pre-geodesic, if it is $\CI^1$ smooth and admits a parametrization $\alpha : I \to \mathcal M$ such that $\dot{\alpha}(t)\neq 0$ for all $t \in I$ and $t \mapsto \alpha(t)$ is a geodesic.

\begin{lemma}
\label{generic}
For any $p \in \mathbb D \setminus \mho$ there exist null geodesics $\gamma^{+}$ and $\gamma^{-}$ and points $q^{\pm}=(t^{\pm},x'^{\pm}) \in \rotom$, with $t^- < t^+$, such that $\gamma^{\pm}$ goes through $p$ and $q^\pm$ and that the null geodesics $\gamma^+$ and $\gamma^-$ intersect on $[t^{-},t^{+}] \times M$ only at $p$.  
\end{lemma}

\begin{proof}
By the definition of $\mathbb D$, there are points
$\hat{q}^{-}=(\hat{t}^{-},x'^{-})$ and $\hat{q}^{+}=(\hat{t}^{+},x'^{+})$ in $\mho$
such that $\hat{q}^{-} \preccurlyeq  p \preccurlyeq  \hat{q}^{+}$.
Analogous to \cite{KLUI}, we define the earliest observation times,
\[
\begin{aligned}
t^{+}=\inf\{t\in[0,T]\,&|\,\tau(p,(t,x'^+))>0\},\\
t^{-}=\sup\{t \in[0,T]\,&|\,\tau((t,x'^-),p)>0\},
\end{aligned}
\]
and set $q^{\pm}=(t^{\pm},x'^{\pm})$.
As $p \notin \mho$, all the three points ${q}^{-} \preccurlyeq  p \preccurlyeq {q}^{+}$ are distinct. 
Writing $p= (t_p, x_p')$, this implies that $t^- < t_p < t^+$.
Since $\M$ is globally hyperbolic, $\tau$ is continuous and therefore 
$$\tau(p,q^+)=\tau(q^-,p)=0.$$ 
But then \cite[Proposition 10.46]{O'neil} implies that the causal curves from $p$ to $q^+$ and from $q^-$ to $p$ are null pregeodesics. 
In particular, there are null geodesics $\gamma^{\pm}$ going through $p$ and $q^\pm$.

 Observe also that $0 < \hat{t}^{-} \leq t^-$. 
Thus $\hat q_\epsilon^- = (t^- - \epsilon, x'^-) \in \mho$ for small $\epsilon > 0$.
The path from $\hat q_\epsilon^-$ to $p$, consisting of a timelike path from $\hat{q}_\epsilon^-$ to $q^-$ together with $\gamma^-$, 
is not a null pregeodesic. Therefore \cite[Proposition 10.46]{O'neil} implies that $\tau(q_\epsilon^-,p)>0$. The same is true for points near $\hat q_\epsilon$, and we could replace $\hat q_-$ with such a point in the above construction. Therefore we may assume without loss of generality that $\gamma^+$ and $\gamma^-$ are not segments of the same null geodesic.  

To get a contradiction let us assume that the two null geodesics $\gamma^{+}$ and $\gamma^{-}$ intersect at a point $\tilde p = (t,x') \in (t^{-},t^{+}) \times M$ and $\tilde p \ne p$. 
Suppose for the moment that $t_p < t < t^+$.
Then following $\gamma^+$ from $q^+$ to $\tilde p$ and $\gamma^-$ from $\tilde p$ to $p$ gives a causal path from $q^+$ to $p$. As $\gamma^+$ and $\gamma^-$ are not segments of the same geodesic, this path is not pregeodesic. Therefore it follows from \cite[Proposition 10.46]{O'neil} that $\tau(p, q^+)>0$, a contradiction. 
The other scenario $t^- < t < t_p$ can be treated analogously. 

In the case that $\gamma^+$ and $\gamma^-$ intersect at $q^+$ we may replace $q^+$ by another point $\tilde q^+$ on $\gamma^+$, such that $\tilde q^+$ is strictly between $q^+$ and $p$ and $\tilde q^+ \in \mho$. Finally, we make an analogous replacement if $\gamma^+$ and $\gamma^-$ intersect at $q^-$.
\end{proof}

We now proceed as in Section~\ref{linearmin}. Consider a point $p=(t_p,x'_p) \in \mathbb D \setminus \mho$. Let $ \gamma^{(0)} := \gamma^{+}$ and $\gamma^{(1)} := \gamma^{-}$ be null geodesics given by Lemma~\ref{generic} with corresponding points $q^{\pm}$ in the set $\rotom$. We denote by ${\xi^{(0)\sharp}}, {\xi^{(1)\sharp}}\in T_p\M$ the tangent vector to $\gamma^{(0)}$ and $\gamma^{(1)}$, respectively, at the point $p$. We may reparametrize $\gamma^{(j)}$ and choose local coordinates near $p$ so that $\g$ coincides with the Minkowski metric at $p$, and that
$$
{\xi^{(0)\sharp}}=(1,\pm \sqrt{1-\sigma^2},\sigma,\underbrace{0,\ldots,0}_{n-2\,\text{times}}),\quad {\xi^{(1)\sharp}}=(1,1,0,\underbrace{0,\ldots,0}_{n-2\,\text{times}})
$$ 
with $\sigma \in [0,1]$. We now define 
$${\xi^{(2)\sharp}}=(1,\sqrt{1-\tilde{\sigma}^2},\tilde{\sigma},\underbrace{0,\ldots,0}_{n-2\,\text{times}}),\quad {\xi^{(3)\sharp}}=(1,\sqrt{1-\tilde{\sigma}^2},-\tilde{\sigma},\underbrace{0,\ldots,0}_{n-2\,\text{times}}),$$
where $\tilde{\sigma} \in (0,1)$. By \cite[Lemma 1]{CLOP}, we have
\bel{linear dependence f}
\tilde{\sigma}^2 {\xi^{(0)\sharp}} + \underbrace{(2b(\sigma)+\mathcal O(\tilde{\sigma}))}_{{\kappa}_1}{\xi^{(1)\sharp}}+\underbrace{(-b(\sigma)+\mathcal O(\tilde{\sigma}))}_{{\kappa}_2}{\xi^{(2)\sharp}}+\underbrace{(-b(\sigma)+\mathcal O(\tilde{\sigma}))}_{{\kappa_3}}{\xi^{(3)\sharp}}=0,
\ee 
where $b(\sigma)=1\mp \sqrt{1-\sigma^2}$. 
As ${\xi^{(0)\sharp}} \ne {\xi^{(1)\sharp}}$ by Lemma \ref{generic}, it holds that $b(\sigma) \ne 0$. In particular, 
$\lim_{\tilde{\sigma}\to 0} \kappa_j(\tilde{\sigma})$ finite and non-zero for $j=1,2,3$. We will write also $\kappa_0(\tilde{\sigma}) = \tilde \sigma^2$.

Now denote by $\gamma^{(j)}$, $j=2,3$, the null geodesics with tangent vector ${\xi^{(j)\sharp}}$ at the point $p$. We choose $\tilde{\sigma}$ sufficiently small so that the geodesics $\gamma^{(j)}$ intersect the set $\rotom$ at some points $q^{(j)}$ near $q^{(1)} := q^{-}$. We write also $q^{(0)} := q^{+}$.
To simplify the discussion, we assume that $\kappa_1 > 0$ and that $\kappa_2, \kappa_3 <0$. Other cases can be treated in a similar manner and are omitted for sake of brevity. With these notations, we consider the null geodesics $\gamma^{(j)}$, $j=0,1,2,3$, with Fermi coordinates $z^{(j)}$ and subsequently construct formal Gaussian beams $u_\tau^{(j)}$ of order 
\begin{equation}\label{N_choice}
N\geq \frac{3n}{2}+7,
\end{equation}
and the form
\bel{gaussianf}
\begin{aligned}
u_\tau^{(0)}&=e^{i\kappa_0\tau\phi^{(0)}}a^{(0)}_{\kappa_0\tau}\quad &u_\tau^{(1)}&=e^{i\kappa_1\tau\phi^{(1)}}a^{(1)}_{\kappa_1\tau},\\
u_\tau^{(2)}&=e^{i\kappa_2\tau\bar{\phi}^{(2)}}\bar a^{(2)}_{\kappa_2\tau}\quad &u_\tau^{(3)}&=e^{i\kappa_3\tau\bar{\phi}^{(3)}}\bar{a}^{(3)}_{\kappa_3\tau},
\end{aligned}
\ee
where the functions $\phi^{(j)}, a^{(j)}$ are exactly as in Section~\ref{WKB} with the initial conditions for all ODEs assigned at the points $q^{(j)}$ in the sense of (\ref{init_conds_q}). Writing $\gamma^{(j)}(s_j) = q^{(j)}$ and $\tilde \xi^{(j)\sharp} = \dot \gamma^{(j)}(s_j)$, we consider the source terms $f_{\tau,q^{(j)},\tilde \xi^{(j)}}$ for $j=1,2,3$ and the test function $f_{\tau,q^{(0)},\tilde \xi^{(0)}}^+$.
Moreover, we define again the three parameter family of source $f_{\tau,\epsilon}$ with $\epsilon = (\epsilon_1, \epsilon_2, \epsilon_3)$
by (\ref{f_family_final}).

\subsection{Recovery of $V$}
\label{waveinterac3}

We start by considering the formal Gaussian beams $u_\tau^{(j)}$  given by (\ref{gaussianf}). Let $f_{\epsilon, \tau}$ be the three parameter family constructed in the previous section and recall that $p$ is an arbitrary point in $\mathbb D \setminus \mho$. In this section, we will prove Theorem \ref{t1} by showing that $V(p)$ is determined by $L_V$. 

Repeating the argument in the beginning of Section \ref{minrecovery}, shows that $L_V$ determines the integral 
$$
\mathcal I = \int_{(0,T)\times \R^n} \mathcal U^{(0)}_\tau \mathcal U^{(1)}_\tau \mathcal U^{(2)}_\tau \mathcal U^{(3)}_\tau \, dV_{\g},
$$
where $dV_{\g}=\sqrt{|\g|}\,dt dx^1\ldots\,dx^n$ denotes the volume form on $(\M,\g)$. 
Using the Sobolev embedding, we obtain from \eqref{sourceest1},
\bel{close1}
\|\mathcal U^{(j)}_\tau - \zeta^{(j)}_- u^{(j)}_\tau\|_{\CI((0,T)\times M)}
\lesssim \tau^{-\frac{n+1}{2}-2}, \quad \forall j=1,2,3,
\ee
where $\zeta^{(j)}_-=1$ in $\mathscr J_{+}(q^{(j)})$.
Indeed, the choice (\ref{N_choice}) guarantees that $K>(n+1)/2 + 2$
where $K$ is as in Lemma \ref{lemerror} with $k=n/2+1$.
Analogously, $\mathcal U^{(0)}_\tau$ satisfies the estimate
\bel{close2}
\|\mathcal U^{(0)}_\tau - \zeta^{(0)}_+ u^{(0)}_\tau\|_{\CI((0,T)\times M)} \lesssim \tau^{-\frac{n+1}{2}-2}.
\ee
with $\zeta^{(0)}_+=1$ in $\mathscr J_{-}(q^{(j)})$.
\\

We proceed to asymptotically analyse $\mathcal I$. Applying the estimates \eqref{close1}--\eqref{close2} together with the boundedness of formal Gaussian beams (see Lemma~\ref{lemerror}), we have 
\bel{expf}\tau^{\frac{n+1}{2}}\mathcal I=\tau^{\frac{n+1}{2}}\int_{(0,T)\times M}u_\tau^{(0)}u_\tau^{(1)} u_\tau^{(2)}u_\tau^{(3)}\,dV_{\g}+\mathcal O(\tau^{-2}).\ee

We will use the method of stationary phase to analyse the product of the four formal Gaussian beams in (\ref{expf}), and need the following lemma. In the lemma we choose $\bar d$ to be an auxiliary distance function on $\mathcal M$.

\begin{lemma}
\label{stationaryphase}
Consider the formal Gaussian beams along the geodesics $\gamma^{(k)}$, $k=0,1,2,3$, in (\ref{gaussianf}), and recall these four geodesic intersect at $p$. Recall also that $\kappa_0, \kappa_1 > 0$ while $\kappa_2, \kappa_3 < 0$. Then the function
$$S:= \kappa_0 \phi^{(0)}+\kappa_1 \phi^{(1)}+\kappa_2 \bar\phi^{(2)}+\kappa_3 \bar{\phi}^{(3)}$$
is well-defined in a small neighborhood of the point $p$ and it holds that
\begin{itemize}
\item[(i)]{$S(p)=0$.}
\item[(ii)]{$\nabla^{\g}S(p)=0$.}
\item[(iii)]{$\Im S(q) \geq m\, \bar{d}(q,p)^2$ for $q$ in a neighborhood of $p$. Here $m>0$ is a constant.}
\end{itemize}
\end{lemma}
\begin{proof}
Note that the first claim is trivial as each of the four phases $\phi^{(k)}$ vanish along $\gamma^{(k)}$ and therefore the sum must vanish at the point of intersection $p$. For the second claim, we note that equation \eqref{phi01} applies to show that along each null geodesic $\gamma^{(k)}$ we have $\nabla^{\g} \phi^{(k)}|_{\gamma^{(k)}}=\dot \gamma^{(k)}$. 
Together with \eqref{linear dependence f}, the second claim follows. 

Let us now consider the last claim. Note that it suffices to show that 
$$ D^2\Im S(X,X)>0\quad \forall X \in T_p\M \setminus 0.$$
First note that $\Im S= \sum_{k=0}^3 |\kappa_k| \Im \phi^{(k)} $ implying that $D^2\Im S(X,X)\geq 0$. Indeed, using the Fermi coordinates we see that for each $k=0,1,2,3$,
\begin{align*}
D^2\Im \phi^{(k)}(X,X) &\geq 0 \quad \forall X \in T_p \M,
\\
D^2 \Im \phi^{(k)}(X,X) &>0\quad \forall X \in T_p\M\setminus \spn{\xi^{(k)\sharp}}
\end{align*}
due to (\ref{phi01})--(\ref{positivity}) and the fact that the Christoffel symbols vanish on $\gamma^{(k)}$ in these coordinates (see Lemma \ref{fermi}). 
Since $\xi^{(0)\sharp}$ and $\xi^{(1)\sharp}$ are linearly independent the claim follows.
\end{proof}

We know from Lemma~\ref{generic} that $p$ is the only point of intersection of the four null geodesics $\gamma^{(k)}$, $k=0,1,2,3$.
Thus the product 
$$
u_\tau^{(0)}u_\tau^{(1)} u_\tau^{(2)}u_\tau^{(3)} = e^{i\tau S} a_{\kappa_0\tau}^{(0)}a_{\kappa_1\tau}^{(1)}\bar{a}_{\kappa_2\tau}^{(2)}\bar{a}_{\kappa_3\tau}^{(3)}
$$ 
is supported in a small neighbourhood of $p$. We expand the amplitudes $a_{\kappa_j\tau}^{(j)}$ in terms of the functions $v_k^{(j)}$ as in (\ref{phase-amplitude}), and apply the method of stationary phase (see e.g. Theorem 7.7.5 in \cite{Ho1}) to (\ref{expf}), term-wise after this expansion. This gives
$$
\tau^{\frac{n+1}{2}}\mathcal I= c_{-1} + \tau^{-1} (c_0 v_1^{(0)}(p) + c_1 v_1^{(1)}(p) + c_2 \bar v_1^{(2)}(p) + c_3 \bar v_1^{(3)}(p)) + \mathcal O(\tau^{-2}),
$$
where $c_j$, $j=-1,0,1,2,3$, are non-zero constants that do not depend on the potential $V$. Using the splitting $v_{1,0}^{(j)} = b_{1,0}^{(j)} + c_{1,0}^{(j)}$ as in (\ref{v_1}) together with the fact that $b_{1,0}^{(j)}$ does not depend on $V$, we see that the map $L_V$ uniquely determines the expressions
\[
\sum_{j=0}^1\frac{c_j}{\kappa_j} \int_{0}^{s_j} V(\gamma^{(j)}(s))\,ds + \sum_{j=2}^3\frac{c_j}{\kappa_j} \int_{0}^{s_j} \bar{V}(\gamma^{(j)}(s))\,ds,
\] 
where $\gamma^{(j)}(0) = q^{(j)}$ and $\gamma^{(j)}(s_j) = p$. 
Finally, noting that $\lim_{\tilde{\sigma}\to 0}\kappa_j \neq 0$ for $j=1,2,3$ and $\lim_{\tilde{\sigma}\to 0} \kappa_0=0$, we deduce that the knowledge of the source to solutions map $L_V$ uniquely determines the integral
$$\int_{0}^{s_0}V(\gamma^{(0)}(s))\,ds.$$
Differentiating in $s_0$ gives $V(p)$, and this completes the proof of Theorem~\ref{t1}.

\bigskip
\paragraph{\bf Acknowledgements}
A.F was supported by EPSRC grant EP/P01593X/1. L.O was supported by EPSRC grants EP/R002207/1 and EP/P01593X/1.

\end{document}

%% file: cylinder_3pts.pdf_tex
\begingroup%
  \makeatletter%
  \providecommand\color[2][]{%
    \errmessage{(Inkscape) Color is used for the text in Inkscape, but the package 'color.sty' is not loaded}%
    \renewcommand\color[2][]{}%
  }%
  \providecommand\transparent[1]{%
    \errmessage{(Inkscape) Transparency is used (non-zero) for the text in Inkscape, but the package 'transparent.sty' is not loaded}%
    \renewcommand\transparent[1]{}%
  }%
  \providecommand\rotatebox[2]{#2}%
  \newcommand*\fsize{\dimexpr\f@size pt\relax}%
  \newcommand*\lineheight[1]{\fontsize{\fsize}{#1\fsize}\selectfont}%
  \ifx\svgwidth\undefined%
    \setlength{\unitlength}{314.65880585bp}%
    \ifx\svgscale\undefined%
      \relax%
    \else%
      \setlength{\unitlength}{\unitlength * \real{\svgscale}}%
    \fi%
  \else%
    \setlength{\unitlength}{\svgwidth}%
  \fi%
  \global\let\svgwidth\undefined%
  \global\let\svgscale\undefined%
  \makeatother%
  \begin{picture}(1,0.88160488)%
    \lineheight{1}%
    \setlength\tabcolsep{0pt}%
    \put(0,0){\includegraphics[width=\unitlength,page=1]{cylinder_3pts.pdf}}%
    \put(0.44997496,0.12123226){\color[rgb]{0,0,0}\makebox(0,0)[lt]{\lineheight{0}\smash{\begin{tabular}[t]{l}$q^-$\end{tabular}}}}%
    \put(0.78128315,0.43202768){\color[rgb]{0,0,0}\makebox(0,0)[lt]{\lineheight{0}\smash{\begin{tabular}[t]{l}$p$\end{tabular}}}}%
    \put(0,0){\includegraphics[width=\unitlength,page=2]{cylinder_3pts.pdf}}%
    \put(0.44859173,0.73754796){\color[rgb]{0,0,0}\makebox(0,0)[lt]{\lineheight{0}\smash{\begin{tabular}[t]{l}$q^+$\end{tabular}}}}%
  \end{picture}%
\endgroup%